\documentclass[article, 11pt]{amsart}

\usepackage{amsmath, amssymb, amsfonts, amsthm, latexsym}
\usepackage{tikz}

\newtheorem{Theorem}{Theorem}[section]
\newtheorem{Lemma}[Theorem]{Lemma}
\newtheorem{Corollary}[Theorem]{Corollary}
\newtheorem{Proposition}[Theorem]{Proposition}

\theoremstyle{definition}
\newtheorem{Remark}[Theorem]{Remark}

\newtheorem{Definition}[Theorem]{Definition}

\def\depth{\operatorname{depth}}

\def\gr{\operatorname{gr}}

\def\ker{\operatorname{ker}}
\def\coker{\operatorname{coker}}

\def\dim{\operatorname{dim}}
\def\id{\operatorname{id}}
\def\pd{\operatorname{pd}}

\def\mod{\operatorname{mod}}

\def\Ext{\operatorname{Ext}}
\def\Tor{\operatorname{Tor}}
\def\Hom{\operatorname{Hom}}

\def\int{\operatorname{in}}

\def\lim{\operatorname{lim}}

\def\ar{\operatorname{ar}}

\def\mm{{\frak m}}

\begin{document}
	
\title[On the Stability of Bass and Betti Numbers under Ideal Perturbations in a Local Ring]{On the Stability of Bass and Betti Numbers under Ideal Perturbations in a Local Ring}
\author{Van Duc Trung}
\address{Department of Mathematics, University of Education, Hue University, 34 Le Loi, Hue, Vietnam}
\email{vdtrung@hueuni.edu.vn}
	
\begin{abstract} Let $(R,\mathfrak{m})$ be a Noetherian local ring, and let $J$ be an arbitrary ideal of $R$. Suppose $M$ is a finitely generated $R$-module. Let $x_1,\ldots,x_r$ be a $J$-filter regular sequence on $M$.  We provide an explicit number $N$ such that the Bass and Betti numbers of $M/(x_1, \ldots, x_r)M$ are preserved when we perturb the sequence $x_1, \ldots,x_r$ by $\varepsilon_1, \ldots, \varepsilon_r \in J^N$.
\end{abstract}
	
\maketitle
\section{Introduction}

This work is motivated by recent advances in the study of small perturbations of algebraic structures, a topic that has received increasing attention due to its relevance in deformation theory. Perturbing defining data by high-order terms naturally arises in many contexts, such as deforming singularities or analyzing the robustness of algebraic invariants under small changes. A typical approach involves adding elements from high powers of the maximal ideal to generators of ideals or sequences, thereby approximating analytic structures by algebraic ones through truncation.

Let $(R, \mm)$ be a Noetherian local ring, and let $I = (x_1, \ldots, x_r)$ be an ideal of $R$. For a positive integer $N > 0$, a perturbation of $I$ with respect to $N$ is an ideal of the form
\[
I' = (x_1 + \varepsilon_1, \ldots, x_r + \varepsilon_r),
\]
where each $\varepsilon_i$ is an arbitrary element of $\mm^N$, for $1 \leq i \leq r$. When $N$ is large, the sequence $(x_1 + \varepsilon_1, \ldots, x_r + \varepsilon_r)$ can be viewed as a small deformation of the original generators of $I$. 

This leads naturally to the following fundamental question:

\medskip
\noindent
\textbf{Question.} \emph{Which structural properties and algebraic invariants of the ideal $I$ or of the quotient ring $R/I$ are preserved under such perturbations, provided that $N$ is sufficiently large?}
\medskip

This question has been studied from several perspectives. Eisenbud~\cite{E} showed that the homology of complexes is well-behaved under small perturbations; for example, regular sequences remain regular when perturbed by high-order terms. Huneke and Trivedi~\cite{HT} extended this to filter regular sequences—an important generalization of regular sequences defined via support conditions. Srinivas and Trivedi~\cite{ST1, ST2} later investigated the behavior of Hilbert–Samuel functions under perturbation, and their results were generalized by Ma, Quy, and Smirnov~\cite{MQS}, and further refined in~\cite{QDT, QT}.

Recent research has also focused on the stability of numerical invariants under perturbations. Ma, Quy, and Smirnov~\cite{MQS} introduced the Hilbert perturbation index, which quantifies the minimal order of perturbation required to preserve the Hilbert–Samuel function. Building on this framework, Quy and the author obtained effective bounds for this index in the generalized Cohen–Macaulay case~\cite{QDT}. Quy and N. V. Trung \cite{QT} later extended the results to arbitrary local rings using a new method.

While considerable attention has been devoted to Hilbert functions and regular sequences, the behavior of homological invariants such as Bass numbers and Betti numbers under small perturbations remains relatively unexplored. Koszul complexes—central objects in commutative algebra and algebraic geometry—are particularly sensitive to such perturbations (see, for instance,~\cite{VDT}). Understanding the extent to which their homology and associated derived functors remain stable under deformations is thus both fundamental and compelling. Related developments appear in the work of Luís Duarte~\cite{LD2}, who studies the behavior of local cohomology modules under small perturbations of ideals and shows that certain cohomological invariants remain stable under such deformations.

The aim of this paper is to study the effect of small perturbations on Betti numbers and Bass numbers associated with quotients by filter regular sequences. This problem is closely related to results of Luís Duarte~\cite{LD1}, where Betti numbers were studied under small perturbations. In particular, it is shown in~\cite{LD1} that a prescribed finite number of Betti numbers can be preserved under sufficiently small perturbations that preserve the Hilbert function; this situation includes, in particular, the case of ideals generated by filter regular sequences. A related and significant question was also posed by Luís Duarte in~\cite[Question~3.9]{LD1}, which asks whether, for a given ideal $I$ of a Noetherian local ring $R$, there exists an integer $N > 0$ such that for every ideal $J$ satisfying $J \equiv I \mod \mathfrak{m}^N$ and having the same Hilbert function as $I$, all Betti numbers of $R/I$ and $R/J$ coincide. The main results of this paper strengthen these earlier results by showing that, in the case of filter regular sequences, one can in fact preserve the entire sequence of Betti numbers under sufficiently small perturbations.

More precisely, using the techniques developed in~\cite{QT} for handling perturbations of filter regular sequences, we give an affirmative answer to the question posed in~\cite[Question~3.9]{LD1} in the case where $I$ is generated by a filter regular sequence on a finitely generated $R$-module $M$. Our main result (Theorem~\ref{main-theorem}) provides an explicit bound on $N$, depending on the Artin--Rees numbers and the Loewy lengths of certain annihilator modules, such that all Betti numbers of the quotient $M/(x_1, \ldots, x_r)M$ are preserved under any perturbation of the form $x_i' = x_i + \varepsilon_i$ with $\varepsilon_i \in \mathfrak{m}^N$. Moreover, we establish a corresponding result for Bass numbers.

By the Artin--Rees Lemma, for any submodule \( N \subseteq M \), there exists an integer $c$ such that for every \( n \geq c \), we have
\[
J^n M \cap N = J^{n - c}(J^c M \cap N).
\]
The least such number \( c \) is called the Artin--Rees number of \( N \) with respect to \( J \), and is denoted by \( \ar_J(N) \).

Let $J$ be an arbitrary ideal of $R$. For any \( R \)-module \( M \), the $J$-Loewy length of \( M \) is defined as
\[
a_J(M) := \inf\{ n \in \mathbb{N} \mid J^n M = 0 \}.
\]

Let \(k\) denote the residue field of a Noetherian local ring \((R, \mathfrak{m})\). For any non-zero finitely generated \(R\)-module \(M\), the \(i\)th Betti number of \(M\), denoted by \(\beta^R_i(M)\), is defined as the \(k\)-dimension of the Tor module \(\Tor^R_i(k, M)\), regarded as a \(k\)-vector space. Similarly, the \(i\)th Bass number of \(M\), denoted by \(\mu^i_R(M)\), is defined as the \(k\)-dimension of the Ext module \(\Ext^i_R(k, M)\), also regarded as a \(k\)-vector space.

Our main result is as follows:

\medskip
\noindent
\textbf{Main Theorem} (Theorem~\ref{main-theorem}).  
\emph{Let $x_1, \ldots, x_r$ be a $J$-filter regular sequence on a finitely generated $R$-module $M$. For each \( i = 1, \ldots, r \), define}
\[
a_i = a_J\left( \frac{(x_1, \ldots, x_{i-1})M : x_i}{(x_1, \ldots, x_{i-1})M} \right),
\]
\emph{and let}
\[
N = \max\left\{ a_1 + 2a_2 + \cdots + 2^{r-1}a_r,\ \ar_J(x_1M), \ldots, \ar_J((x_1, \ldots, x_r)M) \right\} + 2.
\]
\emph{For every \( \varepsilon_1, \ldots, \varepsilon_r \in J^N \), let \( x_i' = x_i + \varepsilon_i \) for \( i = 1, \ldots, r \). Then, for all $j \geq 0$, the following equalities hold:}
\[
\beta_j^R\left( M/(x_1, \ldots, x_r)M \right) = \beta_j^R\left( M/(x_1', \ldots, x_r')M \right),
\]
\[
\mu^j_R\left( M/(x_1, \ldots, x_r)M \right) = \mu^j_R\left( M/(x_1', \ldots, x_r')M \right).
\]

\medskip

The structure of the paper is as follows. In Section~2, we recall foundational concepts on initial modules and filter regular sequences. Section~3 presents the main results.

\section{Preliminaries}
	
Throughout this paper, we assume $(R, \mathfrak{m})$ is a local ring with residue field \( k = R / \mathfrak{m} \). and $J$ is an arbitrary ideal of $R$. 

In this section, we reinterpret the notation and results from \cite[Section~2]{QT} in the context of \( R \)-modules. The proofs follow by an entirely analogous argument. Suppose \( M \) is a finitely generated \( R \)-module. For each element \( x \in M \setminus \{ 0 \} \), denote by \( x^* \) the initial form of \( x \) in the associated graded module \( \gr_J(M) \). For convenience, we set \( 0^* = 0 \). 

For any submodule \( N \subseteq M \), define \( \int(N) \) to be the submodule of \( \gr_J(M) \) generated by all elements \( x^* \), where \( x \in N \). We call \( \int(N) \) the initial module of \( N \) in \( \gr_J(M) \).
\begin{Lemma} \textup{(See \cite[Lemma 2.2]{QT})} \label{int quoetient}
Let \( K \subseteq N \) be submodules of \( M \). Denote by \( \int(N/K) \) the initial module of \( N/K \) in \( \gr_J(M/K) \). Then
\[
\int(N/K) = \int(N)/\int(K).
\]
\end{Lemma}

For any graded submodule \( Q \) of \( \gr_J(M) \), denote by \( d(Q) \) the maximum degree among the elements of a graded minimal generating set of \( Q \).

\begin{Proposition}\label{ar and int} \textup{(See \cite[Proposition 2.3]{QT})} For any submodule \( N \subseteq M \), we have
	\[
	\ar_J(N) = d(\int(N)).
	\]
\end{Proposition}

The next result concerns the Artin–Rees number of submodules in quotient modules.
\begin{Lemma}\textup{(See \cite[Lemma 2.6]{QT})} \label{ar quotient}
Let \( K \subseteq N \) be submodules of \( M \). Denote $\overline{N} = N/K$, then
$$\ar_J(\overline{N}) \leq \ar_J(N).$$	
\end{Lemma}

\begin{Definition}
	Let \( M \) be a finitely generated \( R \)-module. A sequence \( x_1,\dots,x_r \in R \) is called a \( J \)-filter regular sequence on \( M \) if, for each \( i = 1,\dots,r \), the element \( x_i \) is not contained in any associated prime of \( M/(x_1,\dots,x_{i-1})M \) that is not contained in \( V(J) \).
\end{Definition}

\noindent
Note that a sequence \( x_1,\dots,x_r \in R \) is \( J \)-filter regular on \( M \) if and only if
\[
a_J\left( \frac{(x_1,\dots,x_{i-1})M : x_i}{(x_1,\dots,x_{i-1})M} \right) < \infty
\quad \text{for every } 1 \leq i \leq r.
\]

For $r=1$ we have the following result.

\begin{Proposition}\label{one element} \textup{(See \cite[Proposition 3.2]{QT})}
	Suppose $M$ is a finitely generated $R$-module, and let $x$ be a $J$-filter regular element on $M$. Set
	\[
	N = \max\left\{ a_J(0 :_M x),\ \ar_J(xM) + 1 \right\}.
	\]
	For every $\varepsilon \in J^N$, denote $x' = x + \varepsilon$. Then:
	\begin{enumerate}
		\item[\rm (i)] $x'$ is a $J$-filter regular element on $M$, and
		\[
		(0 :_M x) = (0 :_M x').
		\]
		
		\item[\rm (ii)] $\int(xM) = \int(x'M)$.
	\end{enumerate}
\end{Proposition}

\begin{Remark}\label{multiply in Hom}
	With the notation as in Proposition~\ref{one element}, let \( \overline{M} = M / (0 :_M x) \). Then, for every \( \varepsilon \in J^N \), the map
	\[
	\begin{aligned}
		\varphi_\varepsilon : \overline{M} &\longrightarrow M \\
		\overline{m} &\longmapsto \varepsilon m
	\end{aligned}
	\]
	is a well-defined element of \( \Hom_R(\overline{M}, M) \).
\end{Remark}

For a $J$-filter regular sequence of length two, we have the following result.
\begin{Proposition}\label{two elements} \textup{(See \cite[Proposition 3.4]{QT})}
Suppose $M$ is a finitely generated $R$-module, and let $x_1, x_2$ be a $J$-filter regular sequence on $M$. Define
$$a_1 = a_J(0:_M x_1), \quad a_2 = a_J\left( \frac{x_1M : x_2}{x_1M} \right).$$
Set 
$$N = \max\{ a_1 + a_2, \ar_J(x_1M), \ar_J((x_1,x_2)M) \} + 1.$$
For every $\varepsilon \in J^N$, let $x_1' = x_1 + \varepsilon$. Then
\begin{enumerate}
	\item[\rm (i)] $x_1', x_2$ is a $J$-filter regular sequence on $M$ with $a_J(0:_M x_1') = a_1$, and
	 $$a_J\left( \frac{x_1'M : x_2}{x_1'M} \right) \leq 2a_2.$$
     \item[\rm (ii)] $\int((x_1,x_2)M) = \int((x_1',x_2)M)$.
\end{enumerate}
\end{Proposition}

We conclude this section by recalling some well-known properties of Betti and Bass numbers. 
Let $M$ be a finitely generated $R$-module. The projective dimension and the injective dimension of $M$ can be characterized in terms of its Betti and Bass numbers by
\[
\pd_R(M)=\sup\{\, i \mid \beta_i^R(M)>0 \,\}
\quad\text{and}\quad
\id_R(M)=\sup\{\, i \mid \mu_R^i(M)>0 \,\}.
\]
Moreover, the depth of $M$ is given by
\[
\depth_R(M)=\inf\{\, i \mid \mu_R^i(M)>0 \,\}.
\]
These standard characterizations will be used throughout the paper.

\section{Main results}
This section is devoted to proving that Bass numbers and Betti numbers remain invariant under small perturbations. We begin by establishing several foundational lemmas required for the main result.

\begin{Lemma}\label{TorExt}
	Suppose \( M \) is a finitely generated \( R \)-module, and let \( x \in R \) be a \( J \)-filter regular element on \( M \). Denote $\overline{M} = M / (0 :_M x)$, and set
	\[
	N = \max\left\{ a_J(0 :_M x),\ \ar_J(xM) + 1 \right\} + 1.
	\]
	Let \( E \) be any finitely generated \( R \)-module. For every \( \varepsilon \in J^N \), define the map
	\[
	\begin{aligned}
		\varphi_{\varepsilon} : \overline{M} &\longrightarrow M \\
		\overline{m} &\longmapsto \varepsilon m .
	\end{aligned}
	\]
	Then, for every \( i \geq 0 \), the following hold:
	\begin{enumerate}
		\item[\rm (i)] The induced map
		\[
		(\varphi_{\varepsilon})^i_* : \Tor^R_i(E, \overline{M}) \longrightarrow \Tor^R_i(E, M)
		\]
		has image contained in \( J \cdot \Tor^R_i(E, M) \).
		\item[\rm (ii)] The induced map
		\[
		(\varphi_{\varepsilon})^*_i : \Ext^i_R(E, \overline{M}) \longrightarrow \Ext^i_R(E, M)
		\]
		has image contained in \( J \cdot \Ext^i_R(E, M) \).
	\end{enumerate}
\end{Lemma}

\begin{proof}
	Recall that for each \( i \geq 0 \), the assignments
	\[
	\Hom_R(\overline{M}, M) \longrightarrow 
	\Hom_R\!\left(\Tor^R_i(E, \overline{M}), \Tor^R_i(E, M)\right),
	\quad \varphi \longmapsto \varphi^i_*,
	\]
	and
	\[
	\Hom_R(\overline{M}, M) \longrightarrow 
	\Hom_R\!\left(\Ext^i_R(E, \overline{M}), \Ext^i_R(E, M)\right),
	\quad \varphi \longmapsto \varphi^*_i,
	\]
	are \( R \)-linear homomorphisms.
	
	Write
	\[
	\varepsilon = s_1 t_1 + \cdots + s_n t_n,
	\]
	with \( s_j \in J \) and \( t_j \in J^{N-1} \) for \( j = 1, \ldots, n \).
	Each \( t_j \) induces a well-defined homomorphism
	\(
	\varphi_{t_j} \in \Hom_R(\overline{M}, M).
	\)
	By \( R \)-linearity, we obtain
	\[
	(\varphi_{\varepsilon})^i_* 
	= s_1 (\varphi_{t_1})^i_* + \cdots + s_n (\varphi_{t_n})^i_* .
	\]
	Hence the image of \( (\varphi_{\varepsilon})^i_* \) is contained in
	\( J \cdot \Tor^R_i(E, M) \), proving {\rm (i)}.
	
	The proof of {\rm (ii)} is completely analogous, using the \( R \)-linearity of the induced maps on Ext.
\end{proof}

From Lemma~\ref{TorExt}, we obtain the following proposition.

\begin{Proposition}\label{induced maps}
	Suppose \( M \) is a finitely generated \( R \)-module, and let \( x \in R \) be a \( J \)-filter regular element on \( M \). Denote \( \overline{M} = M / (0 :_M x) \), and set
	\[
	N = \max\left\{ a_J(0 :_M x),\ \ar_J(xM) + 1 \right\} + 1.
	\]
	For every \( \varepsilon \in J^N \), define the maps
	\[
	\begin{aligned}
		\varphi  : \overline{M} &\longrightarrow M, & \quad \overline{m} &\longmapsto x m, \\
		\varphi': \overline{M} &\longrightarrow M, & \quad \overline{m} &\longmapsto (x + \varepsilon) m.
	\end{aligned}
	\]
	
	Then, for every \( i \geq 0 \), the following hold:
	\begin{enumerate}
		\item[\rm (i)] The induced maps
		\[
		\varphi^i_*: \Tor^R_i(k,\overline{M}) \rightarrow \Tor^R_i(k,M) \quad \text{and} \quad (\varphi')^i_*: \Tor^R_i(k,\overline{M}) \rightarrow \Tor^R_i(k,M)
		\]
		coincide.
		
		\item[\rm (ii)] The induced maps
		\[
		\varphi^*_i: \Ext^i_R(k,\overline{M}) \rightarrow \Ext^i_R(k,M) \quad \text{and} \quad (\varphi')^*_i: \Ext^i_R(k,\overline{M}) \rightarrow \Ext^i_R(k,M)
		\]
		coincide.
	\end{enumerate}
\end{Proposition}

\begin{proof}
	We have
	\[
	\begin{aligned}
		(\varphi' - \varphi) : \overline{M} &\longrightarrow M \\
		\overline{m} &\longmapsto \varepsilon m.
	\end{aligned}
	\]
	By Lemma~\ref{TorExt}, the induced map
	\[
	(\varphi' - \varphi)^i_* : \Tor^R_i(k, \overline{M}) \longrightarrow \Tor^R_i(k, M)
	\]
	has image contained in \( J \cdot \Tor^R_i(k, M) \).
	
	Now, \( \Tor^R_i(k, M) \) is a \( k \)-vector space and \( J \subseteq \mathfrak{m} \), so
	\[
	J \cdot \Tor^R_i(k, M) = 0.
	\]
	Hence, \( (\varphi' - \varphi)^i_* = 0 \), which implies \( (\varphi')^i_* = \varphi^i_* \). This proves part (i).
	
	For part (ii), the same argument applies to the induced maps on \( \Ext \). That is, the map
	\[
	(\varphi' - \varphi)^*_i : \Ext^i_R(k, \overline{M}) \longrightarrow \Ext^i_R(k, M)
	\]
	has image contained in \( J \cdot \Ext^i_R(k, M) \), which is also zero since \( \Ext^i_R(k, M) \) is a \( k \)-vector space. Therefore,
	\[
	(\varphi')^*_i = \varphi^*_i.
	\]
\end{proof}

For a \( J \)-filter regular element, we have the following result:
\begin{Proposition}\label{length one}
	Suppose \( M \) is a finitely generated \( R \)-module, and let \( x \in R \) be a \( J \)-filter regular element on \( M \). Set
	\[
	N = \max\left\{ a_J(0 :_M x),\ \ar_J(xM) + 1 \right\} + 1.
	\]
	For every \( \varepsilon \in J^N \), let \( x' = x + \varepsilon \). Then, for each \( j \geq 0 \), we have:
	\begin{enumerate}\setlength{\itemsep}{0.3em}
		\item[\rm (i)] \( \beta^R_j(M/xM) = \beta^R_j(M/x'M) \),
		\item[\rm (ii)] \( \mu^j_R(M/xM) = \mu^j_R(M/x'M) \).
	\end{enumerate}
\end{Proposition}

\begin{proof}
   We first consider the short exact sequence
	\[
	0 \longrightarrow \overline{M} \xrightarrow{\varphi} M \longrightarrow M/xM \longrightarrow 0,
	\]
	where \( \overline{M} = M / (0 :_M x) \) and
	\[
	\begin{aligned}
		\varphi : \overline{M} &\longrightarrow M \\
		\overline{m} &\longmapsto x m.
	\end{aligned}
	\]
	Applying the functor \( \Tor^R_i(k, -) \), we obtain the long exact sequence:
	\begin{align*}
		\cdots \longrightarrow \Tor^R_2(k, M/xM) \longrightarrow \Tor^R_1(k, \overline{M}) 
		&\xrightarrow{\varphi^1_*} \Tor^R_1(k, M) \longrightarrow \Tor^R_1(k, M/xM) \\
		&\longrightarrow k \otimes_R \overline{M} \xrightarrow{\varphi^0_*} k \otimes_R M 
		\longrightarrow k \otimes_R (M/xM) \longrightarrow 0.
	\end{align*}
	From this, for each \( i \geq 0 \), we obtain a short exact sequence
	\[
	0 \longrightarrow \coker \varphi^i_* \longrightarrow \Tor^R_i(k, M/xM) \longrightarrow \ker \varphi^{i-1}_* \longrightarrow 0,
	\]
	where we interpret \( \ker \varphi^{-1}_* = 0 \) when \( i = 0 \).
	
	Similarly, from the short exact sequence
	\[
	0 \longrightarrow \overline{M} \xrightarrow{\varphi'} M \longrightarrow M/x'M \longrightarrow 0,
	\]
	with
	\[
	\begin{aligned}
		\varphi' : \overline{M} &\longrightarrow M \\
		\overline{m} &\longmapsto (x + \varepsilon) m,
	\end{aligned}
	\]
	we obtain the analogous short exact sequence:
	\[
	0 \longrightarrow \coker (\varphi')^i_* \longrightarrow \Tor^R_i(k, M/x'M) \longrightarrow \ker (\varphi')^{i-1}_* \longrightarrow 0.
	\]
	
	By Proposition~\ref{induced maps}, we have \( \varphi^i_* = (\varphi')^i_* \) for all \( i \geq 0 \). It follows that
	\[
	\dim_k \Tor^R_i(k, M/xM) = \dim_k \Tor^R_i(k, M/x'M),
	\]
	which proves part (i).
	
	For part (ii), we apply the contravariant functor \( \Hom_R(k, -) \) to the same short exact sequences. From
	\[
	0 \longrightarrow \overline{M} \xrightarrow{\varphi} M \longrightarrow M/xM \longrightarrow 0,
	\]
	we obtain the long exact sequence:
	\begin{align*}
		\cdots \longrightarrow \Ext^{i-1}_R(k, M/xM) \longrightarrow \Ext^i_R(k, \overline{M}) 
		&\xrightarrow{\varphi^*_i} \Ext^i_R(k, M) \longrightarrow \Ext^i_R(k, M/xM) \\
		&\longrightarrow \Ext^{i+1}_R(k, \overline{M}) \longrightarrow \cdots
	\end{align*}
	which yields the short exact sequence:
	\[
	0 \longrightarrow \coker \varphi^*_i \longrightarrow \Ext^i_R(k, M/xM) \longrightarrow \ker \varphi^*_{i+1} \longrightarrow 0.
	\]
	
	Similarly, applying \( \Hom_R(k, -) \) to the sequence for \( x' \), we get:
	\[
	0 \longrightarrow \coker (\varphi')^*_i \longrightarrow \Ext^i_R(k, M/x'M) \longrightarrow \ker (\varphi')^*_{i+1} \longrightarrow 0.
	\]
	
	Again, by Proposition~\ref{induced maps}, we have \( \varphi^*_i = (\varphi')^*_i \) for all \( i \geq 0 \). It follows that
	\[
	\dim_k \Ext^i_R(k, M/xM) = \dim_k \Ext^i_R(k, M/x'M),
	\]
	and so \( \mu^i_R(M/xM) = \mu^i_R(M/x'M) \), completing the proof.
\end{proof}
Before extending the above result to \( J \)-filter regular sequences of arbitrary length, we establish the following auxiliary lemmas.
\begin{Lemma}\label{cong}
	Suppose \( M \) is a finitely generated \( R \)-module, and let \( x_1, \ldots, x_r \) be a sequence in \( R \), where \( r \geq 2 \). Then, for every \( 2 \leq i \leq r \), we have the following isomorphism:
	\[
	\frac{(x_1,\ldots,x_{i-1})M : x_i}{(x_1,\ldots,x_{i-1})M + \left((x_2,\ldots,x_{i-1})M : x_i\right)} 
	\cong 
	\frac{(x_2,\ldots,x_i)M : x_1}{(x_2,\ldots,x_i)M + \left((x_2,\ldots,x_{i-1})M : x_1\right)},
	\]
	where \( (x_2,\ldots,x_{i-1}) = 0 \) when \( i = 2 \).
\end{Lemma}
\begin{proof}
The case \( i = 2 \) follows by an argument analogous to that of \cite[Lemma~3.3]{QT}.

For \( i \geq 3 \), set $\overline{M} = M/(x_2, \ldots, x_{i-1})M$. Applying the statement to the module \( \overline{M} \) and the elements \( x_1 \) and \( x_i \), we obtain the following isomorphism:
\[
\frac{x_1 \overline{M} : x_i}{x_1 \overline{M} + \left(0 :_{\overline{M}} x_i \right)} 
\cong 
\frac{x_i \overline{M} : x_1}{x_i \overline{M} + \left(0 :_{\overline{M}} x_1 \right)},
\]
which yields the desired result.
\end{proof}

\begin{Lemma}\label{x1}
	Suppose \( M \) is a finitely generated \( R \)-module, and let \( x_1, \ldots, x_r \in R \) be a \( J \)-filter regular sequence on \( M \), where \( r \geq 2 \). For each \( i = 1, \ldots, r \), define
	\[
	a_i = a_J\left( \frac{(x_1,\ldots,x_{i-1})M : x_i}{(x_1,\ldots,x_{i-1})M} \right).
	\]
	Then, for every \( 2 \leq i \leq r \), the element \( x_1 \) is \( J \)-filter regular on \( M_i := M / (x_2, \ldots, x_i)M \), and we have
	\[
	a_J\left( \frac{(x_2,\ldots,x_i)M : x_1}{(x_2,\ldots,x_i)M} \right) \leq a_1 + \cdots + a_i.
	\]
\end{Lemma}

\begin{proof}
	We proceed by induction on \( i \). For $i=2$, since
	\[
	J^{a_2}(x_1M : x_2) \subseteq x_1M \subseteq x_1M + (0 :_M x_2),
	\]
	it follows from Lemma~\ref{cong} that
	\[
	J^{a_2}(x_2M : x_1) \subseteq x_2M + (0 :_M x_1).
	\]
	Applying \( J^{a_1} \) and using the definition of \( a_1 \), we get
	\[
	J^{a_1 + a_2}(x_2M : x_1) \subseteq x_2M + J^{a_1}(0 :_M x_1) \subseteq x_2M.
	\]
	Thus, \( x_1 \) is \( J \)-filter regular on \( M_2 = M / x_2M \), and
	\[
	a_J\left( \frac{x_2M : x_1}{x_2M} \right) \leq a_1 + a_2.
	\]
	
	Let \( i > 2 \) and set \( s_i = a_1 + \cdots + a_{i-1} \). By the induction hypothesis,
	\[
	a_J\left( \frac{(x_2, \ldots, x_{i-1})M : x_1}{(x_2, \ldots, x_{i-1})M} \right) \leq s_i,
	\]
	which implies
	\[
	J^{s_i}\big((x_2, \ldots, x_{i-1})M : x_1\big) \subseteq (x_2, \ldots, x_{i-1})M.
	\]
	
	From the definition of \( a_i \) and Lemma~\ref{cong}, we have
	\[
	J^{a_i}\big((x_1, \ldots, x_{i-1})M : x_i\big) \subseteq (x_1, \ldots, x_{i-1})M,
	\]
	which yields
	\[
	J^{a_i}\big((x_2, \ldots, x_i)M : x_1\big) \subseteq (x_2, \ldots, x_i)M + \big((x_2, \ldots, x_{i-1})M : x_1\big).
	\]
	
	Applying \( J^{s_i} \) again, we obtain
	\[
	J^{s_i + a_i}\big((x_2, \ldots, x_i)M : x_1\big) \subseteq (x_2, \ldots, x_i)M.
	\]
	Therefore,
	\[
	a_J\left( \frac{(x_2, \ldots, x_i)M : x_1}{(x_2, \ldots, x_i)M} \right) \leq s_i + a_i = a_1 + \cdots + a_i,
	\]
	and hence \( x_1 \) is \( J \)-filter regular on \( M_i = M / (x_2, \ldots, x_i)M \).
\end{proof}

We now present the main result of this paper, which generalizes Proposition~\ref{length one} to the setting of \( J \)-filter regular sequences of arbitrary length.

\begin{Theorem}\label{main-theorem}
	Suppose \( M \) is a finitely generated \( R \)-module, and let \( x_1, \ldots, x_r \in R \) be a \( J \)-filter regular sequence on \( M \). For each \( i = 1, \ldots, r \), define
	\[
	a_i = a_J\left( \frac{(x_1,\ldots,x_{i-1})M : x_i}{(x_1,\ldots,x_{i-1})M} \right).
	\]
	Set
	\[
	N = \max\left\{ a_1 + 2a_2 + \cdots + 2^{r-1}a_r,\ \ar_J(x_1M),\ \ldots,\ \ar_J((x_1,\ldots,x_r)M) \right\} + 2.
	\]
	For every \( \varepsilon_1, \ldots, \varepsilon_r \in J^N \), let \( x_i' = x_i + \varepsilon_i \) for \( i = 1, \ldots, r \). Then, for each \( j \geq 0 \), we have:
	\begin{enumerate}\setlength{\itemsep}{0.5em}
		\item[\rm (i)] \( \beta^R_j\left(M / (x_1, \ldots, x_r)M\right) = \beta^R_j\left(M / (x_1', \ldots, x_r')M\right) \),
		\item[\rm (ii)] \( \mu^j_R\left(M / (x_1, \ldots, x_r)M\right) = \mu^j_R\left(M / (x_1', \ldots, x_r')M\right) \).
	\end{enumerate}
\end{Theorem}
\begin{proof}
We proceed by induction on \( r \). The case \( r = 1 \) follows from Proposition~\ref{length one}.

Now assume \( r \geq 2 \). Define \( M_i := M / (x_2, \ldots, x_i)M \) for $1 \leq i \leq r$, where $M_1=M$. For each \( 2 \leq i \leq r \), by Lemma~\ref{x1}, the element \( x_1 \) is \( J \)-filter regular on $M_i$, and 
\[
a_J\left( 0 :_{M_i} x_1 \right) \leq a_1 + \cdots + a_i.
\]

For each \( 1 \leq i \leq r-1 \), since \( x_{i+1} \) is \( J \)-filter regular on \( M_i / x_1M_i \), it follows that the sequence \( x_1, x_{i+1} \) is \( J \)-filter regular on \( M_i \), and
\[
a_J\left( \frac{x_1 M_i : x_{i+1}}{x_1 M_i} \right) = a_{i+1}.
\]
By Lemma~\ref{ar quotient}, we have:
\[
\ar_J(x_1M_i) \leq \ar_J((x_1,\ldots,x_i)M), \quad \text{and} \quad \ar_J((x_1,x_{i+1})M_i) \leq \ar_J((x_1,\ldots,x_{i+1})M).
\]
Thus,
\[
N \geq \max \left\{ a_J\left( 0 :_{M_i} x_1 \right) + a_J\left( \frac{x_1 M_i : x_{i+1}}{x_1 M_i} \right),\ \ar_J(x_1M_i),\ \ar_J((x_1,x_{i+1})M_i) \right\} +1.
\]
Therefore, applying Proposition~\ref{two elements} to the sequence \( x_1, x_{i+1} \) on \( M_i \), we have that \( x_1', x_{i+1} \) is \( J \)-filter regular on \( M_i \), and
\[
a_J\left( \frac{x_1'M_i : x_{i+1}}{x_1'M_i} \right) \leq 2a_{i+1}.
\]
It follows that the sequence \( x_1', x_2, \ldots, x_r \) is \( J \)-filter regular on \( M \). By Proposition~\ref{two elements}, we also have:
\[
\int((x_1', x_{i+1})M_i) = \int((x_1, x_{i+1})M_i).
\]
Hence, by Lemma~\ref{int quoetient}, it follows that:
\[
\int((x_1', x_2, \ldots, x_{i+1})M) = \int((x_1, x_2, \ldots, x_{i+1})M).
\]
From Proposition~\ref{ar and int}, we conclude:
\[
\ar_J((x_1', x_2, \ldots, x_{i+1})M) = \ar_J((x_1, x_2, \ldots, x_{i+1})M).
\]

Since \( x_1', x_2, \ldots, x_r \) is \( J \)-filter regular on \( M \), it follows that \( x_2, \ldots, x_r \) is \( J \)-filter regular on \( M' := M / x_1'M \). Set:
\[
a_1' := a_J(0 :_{M'} x_2), \quad \text{and} \quad a_i' := a_J\left( \frac{(x_2, \ldots, x_i)M' : x_{i+1}}{(x_2, \ldots, x_i)M'} \right) \quad \text{for } 2 \leq i \leq r - 1.
\]
Note that for $1 \leq i \leq r-1$, we have
$$a_i' = a_J\left( \frac{x_1'M_i : x_{i+1}}{x_1'M_i} \right).$$
By the bounds obtained above, we have:
\[
N \geq \max\left\{ a_1' + 2a_2' + \cdots + 2^{r-2}a_{r-1}',\ \ar_J(x_2M'),\ \ldots,\ \ar_J((x_2, \ldots, x_r)M') \right\} + 2.
\]
Hence, by the induction hypothesis, we obtain:
\[
\beta^R_j\left( M' / (x_2, \ldots, x_r)M' \right) = \beta^R_j\left( M' / (x_2', \ldots, x_r')M' \right),
\]
and
\[
\mu^R_j\left( M' / (x_2, \ldots, x_r)M' \right) = \mu^R_j\left( M' / (x_2', \ldots, x_r')M' \right),
\]
for every \( j \geq 0 \). This is equivalent to:
\[
\beta^R_j\left( M / (x_1', x_2, \ldots, x_r)M \right) = \beta^R_j\left( M / (x_1', x_2', \ldots, x_r')M \right),
\]
and
\[
\mu^R_j\left( M / (x_1', x_2, \ldots, x_r)M \right) = \mu^R_j\left( M / (x_1', x_2', \ldots, x_r')M \right),
\]
for every \( j \geq 0 \).

Now, from Lemma~\ref{x1}, the element \( x_1 \) is \( J \)-filter regular on \( M_r = M / (x_2, \ldots, x_r)M \), and 
\[
a_J(0:_{M_r} x_1) \leq a_1 + \cdots + a_r.
\]
Moreover, \( \ar_J(x_1M_r) \leq \ar_J((x_1,\ldots,x_r)M) \). Hence, by Proposition~\ref{length one}, we have:
\[
\beta^R_j\left( M_r / x_1M_r \right) = \beta^R_j\left( M_r / x_1'M_r \right),
\]
and
\[
\mu^R_j\left( M_r / x_1M_r \right) = \mu^R_j\left( M_r / x_1'M_r \right),
\]
for every \( j \geq 0 \). This is equivalent to:
\[
\beta^R_j\left( M / (x_1, x_2, \ldots, x_r)M \right) = \beta^R_j\left( M / (x_1', x_2, \ldots, x_r)M \right),
\]
and
\[
\mu^R_j\left( M / (x_1, x_2, \ldots, x_r)M \right) = \mu^R_j\left( M / (x_1', x_2, \ldots, x_r)M \right),
\]
for every \( j \geq 0 \).

Combining with the above equality, we obtain the desired result.
\end{proof}

From Theorem~\ref{main-theorem}, we immediately obtain the following corollary.

\begin{Corollary}
	Suppose \( M \) is a finitely generated \( R \)-module, and let \( x_1, \ldots, x_r \in R \) be a \( J \)-filter regular sequence on \( M \). For each \( i = 1, \ldots, r \), define
	\[
	a_i = a_J\left( \frac{(x_1,\ldots,x_{i-1})M : x_i}{(x_1,\ldots,x_{i-1})M} \right).
	\]
	Set
	\[
	N = \max\left\{ a_1 + 2a_2 + \cdots + 2^{r-1}a_r,\ \ar_J(x_1M),\ \ldots,\ \ar_J((x_1,\ldots,x_r)M) \right\} + 2.
	\]
	For every \( \varepsilon_1, \ldots, \varepsilon_r \in J^N \), let \( x_i' = x_i + \varepsilon_i \) for \( i = 1, \ldots, r \). Then:
	\begin{enumerate}\setlength{\itemsep}{0.5em}
		\item[\rm (i)] \( \pd_R\left(M / (x_1, \ldots, x_r)M\right) = \pd_R\left(M / (x_1', \ldots, x_r')M\right) \),
		\item[\rm (ii)] \( \id_R\left(M / (x_1, \ldots, x_r)M\right) = \id_R\left(M / (x_1', \ldots, x_r')M\right) \),
		\item[\rm (iii)] \( \depth_R\left(M / (x_1, \ldots, x_r)M\right) = \depth_R\left(M / (x_1', \ldots, x_r')M\right) \).
	\end{enumerate}	
\end{Corollary}

In \cite{LD1}, it is shown that, up to the finite stage at which Betti numbers are preserved, the minimal free resolution of a perturbed ideal can be obtained as an $\mathfrak m$-adic perturbation of the minimal free resolution of the original ideal. Motivated by Theorem~\ref{main-theorem}, it is natural to ask whether an analogous phenomenon holds in our setting.

More precisely, under the assumptions of Theorem~\ref{main-theorem}, is it true that the 
entire minimal free resolution of
\[
M/(x'_1,\ldots,x'_r)M
\]
can be obtained as a $J$-adic perturbation of the minimal free resolution of
\[
M/(x_1,\ldots,x_r)M?
\]

At present, our methods allow us to control the induced maps on  $\Tor^R_i(k,-)$ and $\Ext^i_R(k,-)$, which is sufficient to establish the invariance of Betti and Bass numbers. However, they do not appear to provide enough information to compare the differentials in the corresponding minimal free resolutions. In particular, while these numerical invariants are preserved under sufficiently small perturbations, it remains unclear whether the entire minimal free resolution admits a compatible $J$-adic deformation.

We believe that a positive answer to this question would lead to a deeper understanding of the homological stability phenomena studied in this paper, and we hope to address this problem in future work.

\noindent \textbf{Acknowledgements.}  This work was initiated during a research stay of the author at Vietnam Institute for Advanced Study in Mathematics (VIASM) in 2023. The author sincerely thanks the referee for a thorough reading of the manuscript and for suggesting valuable improvements. This work was supported by the NAFOSTED (Vietnam) under the grant number 101.04-2023.07. The author also acknowledges the partial support of Hue University under the Core Research Program, Grant No. NCTB.DHH.2024.01.


\begin{thebibliography}{1}

\bibitem{LD1}
L. Duarte, {\em Betti numbers under small perturbations}, J. Algebra {\bf 594}, 138-153 (2022).

\bibitem{LD2}
L. Duarte, \emph{Local cohomology under small perturbations}, Nagoya Math. J. \textbf{258}, 183-200 (2025).


\bibitem{E}
D. Eisenbud, {\em Adic approximation of complexes, and multiplicities}, Nagoya Math. J. {\bf 54}, 61–67 (1974).

\bibitem{HT}
C. Huneke and V. Trivedi, {\em The Height of Ideals and Regular Sequences}, Manus. Math. {\bf 93}, 137-142 (1997).


\bibitem{MQS}
L. Ma, P. H. Quy, and I. Smirnov, {\em Filter regular sequence under small perturbations}, Math. Ann. {\bf 378}, 243-254 (2020).

\bibitem{QT}
P. H. Quy and N. V. Trung, {\em When does a perturbation of the equations preserve the normal cone}, Trans. Amer. Math. Soc. {\bf 376}, 4957-4978 (2023).

\bibitem{QDT}
P. H. Quy and V. D. Trung, {\em Small perturbations in generalized Cohen-Macaulay local rings}, J.
Algebra {\bf 587}, 555–568 (2021).

\bibitem{ST1}
V. Srinivas and V. Trivedi, {\em The Invarience of Hilbert Functions of Quotients under Small Perturbations}, J. Algebra {\bf 186}, 1-19 (1996).

\bibitem{ST2}
V. Srinivas and V. Trivedi, {\em A finiteness theorem for the Hilbert functions of complete intersection local rings}, Math. Z. {\bf 225}, 543-558 (1997).

\bibitem{VDT}
V.~D. Trung, \emph{Koszul homology under small perturbations}, Bull. Korean Math. Soc. {\bf 63}, 77-90 (2026).

\end{thebibliography}
\end{document}